\documentclass{amsart}
\usepackage{hyperref}
\usepackage[T1]{fontenc}
\usepackage[utf8]{inputenc}
\usepackage[english]{babel}
\usepackage{csquotes,newcent}

\usepackage{amsfonts,amsmath,amssymb,bbm}
\usepackage{mathtools}

\usepackage{graphicx}

\newcommand{\cF}{\ensuremath{\mathcal{F}}}

\newcommand{\cL}{\ensuremath{\mathcal{L}}}
\newcommand{\cN}{\ensuremath{\mathcal{N}}}

\newcommand{\bC}{\ensuremath{\mathbb{C}}}
\newcommand{\bE}{\ensuremath{\mathbb{E}}}
\newcommand{\bN}{\ensuremath{\mathbb{N}}}
\newcommand{\bP}{\ensuremath{\mathbb{P}}}
\newcommand{\bR}{\ensuremath{\mathbb{R}}}

\newtheorem{theorem}{Theorem}[section]
\newtheorem{lemma}[theorem]{Lemma}
\newtheorem{proposition}[theorem]{Proposition}
\theoremstyle{remark}
\newtheorem{remark}[theorem]{Remark}

\begin{document}

\title[A class of fractional O-U processes]{A class of fractional Ornstein-Uhlenbeck processes mixed with a Gamma distribution}%

\author{Luigi Amedeo Bianchi}
\address{L.A.B. -- Dipartimento di Matematica, Universit\`a di Trento, via Sommarive 14, 38123 Povo (Trento), Italy}
\email{luigiamedeo.bianchi@unitn.it}

\author{Stefano Bonaccorsi}
\address{S.B. -- Dipartimento di Matematica, Universit\`a di Trento, via Sommarive 14, 38123 Povo (Trento), Italy}
\email{  stefano.bonaccorsi@unitn.it}

\author{Luciano Tubaro}
\address{L.T. -- Dipartimento di Matematica, Universit\`a di Trento, via Sommarive 14, 38123 Povo (Trento), Italy}
\email{  luciano.tubaro@unitn.it}

\begin{abstract}
	We consider a sequence of fractional Ornstein-Uhlenbeck processes, that are defined as solutions of a family of stochastic Volterra equations with kernel given by the Riesz derivative kernel, and leading  coefficients given by a sequence of  independent Gamma random variables.
	We construct a new process by taking  the empirical mean of this sequence.
	In our framework, the processes involved are not Markovian, hence the analysis of their asymptotic behaviour requires some ad hoc construction.
	In our main result, we prove the almost sure convergence in the space of trajectories of the empirical means to a given Gaussian process, which we characterize completely.
\end{abstract}

\maketitle
\thispagestyle{empty}


\section{Introduction}

Recently, there has been a certain interest in the literature about the convergence of the \emph{empirical means} of a sequence of realizations of the solution of a stochastic differential equation (SDE) with random coefficients as a mean to construct new classes of stochastic processes.
For instance, K.~El-Sebaiy and co-authors, in a sequence of papers~\cite{Es-Tud2015F,Es-FarHil2020PMS,DouEs-Tud2020PMD} have constructed several classes of Ornstein-Uhlenbeck type processes led by a fractional Brownian motion $B^H(t)$ or a Hermite process $Z^H_q(t)$ of order $q \ge 1$.

The construction of this kind kind of processes
was introduced by \cite{IglTer1999EJP} in the analysis of a stationary Ornstein-Uhlenbeck process
\begin{equation}\label{eq:0}
	\mathrm{d}X_k(t) = -\alpha_k X_k(t) \, \mathrm{d}t + \mathrm{d}W(t),
\end{equation}
with initial condition $X_k(0) = 0$,
where
$W = \{W(t),\ t \ge 0\}$ is a real standard Brownian motion defined on a stochastic basis $(\Omega, \cF, \{\cF_t,\ t \in \bR\}, \bP)$
and $\{\alpha_k,\ k \ge 1\}$ is a sequence of independent, identically distributed random variables,
independent of $W$. With no loss of generality, we may assume that
the sequence
of the $\alpha_k$ is defined on a different probability space $(\bar\Omega, \bar\cF, \bar\bP)$.

The \emph{empirical means} are defined by
\begin{equation}
	\label{eq:emp}
	Y_n(t) = \frac1n \sum_{k=1}^n X_k(t).
\end{equation}
Under the assumption that the $\alpha_k$ are independent, identically distributed, with a Gamma distribution $\Gamma(\mu, \lambda)$, with shape parameter $\mu > \frac12$ and scale parameter $\lambda > 0$,
in~\cite{IglTer1999EJP} the authors prove that $\bar\bP$-a.s.~the sequence $(Y_n(t))_n$ converges to a process $Y(t)$ for any $t \in \bR$ and in $C([a,b])$ for arbitrary $a < b$, where the process $Y(t)$ is defined in law by
\begin{equation*}
	Y(t) = \int_{0}^t \left( \frac{\lambda}{t - s + \lambda} \right)^\mu \, \mathrm{d}W(s).
\end{equation*}

In order to extend the above construction, we recall the notion of stochastic Volterra processes defined as convolution processes between a given kernel and a Brownian motion $W$
\begin{align*}
	X(t) = \int_0^t g(t,s) \, \mathrm{d}W(s), \qquad t \ge 0.
\end{align*}
Such processes are related to (and somehow coincide with a subclass of) ambit processes, as introduced by Barndorff-Nielsen \cite{BarSch2007SAA,BarBenVer2011AMMfF}, and in particular to the so-called Brownian semi-stationary processes, see~\cite{BarSch2009AFM}.

These processes cannot be expressed in It\^o differential form \eqref{eq:0}, except in the case of an exponential kernel $g(t,s) = e^{\lambda (t-s)}$, due to the presence of a memory term.
It is however possible to describe their behaviour through an integral stochastic (Volterra) equation. 
Although this process is not Markovian, several properties can be studied, see for instance \cite{BonTub2003SAaA} for a survey of known results and the relation with the fractional Brownian motion.

In analogy with the existing literature, we aim to construct a class of stochastic Volterra processes via an aggregation procedure (i.e., taking the limit of the empirical means).
Thanks to this approach, we are able to prove some of its properties.
We shall refer to our processes as \emph{fractional Ornstein-Uhlenbeck} processes, since they are given as solutions of
the integral equation
\begin{equation}
	\label{eq:1}
	X_k(t) = -\alpha_k \left( g_\rho * X_k \right)(t) + W(t), \qquad k \ge 1, \quad t > 0,
\end{equation}
with initial condition $X_k(0) = 0$, where $*$ denotes the convolution product and
\begin{equation*}
	g_\rho(t) = \frac{1}{\Gamma(\rho)} t^{\rho-1}, \qquad t > 0,
\end{equation*}
denotes the fractional integration kernel for any $\rho > 0$, cf.~\cite{Rie1949AM}.
It is well known \cite{BonTub2003SAaA, CleDa1996ADANDLCSFMENRLMEA}
that the solution of \eqref{eq:1} is a Gaussian process defined in
terms of the
scalar resolvent equation
\begin{equation}
	\label{eq:2}
	s_\alpha(t)  + \alpha \int_0^t g_\rho(t-\tau) s_\alpha(\tau) \, d\tau = 1,
	\qquad t \ge 0, \quad \alpha \ge 0,
\end{equation}
by means of the formula
\begin{equation}
	\label{eq:def-Xk}
	X_k(t) = \int_{0}^t s_{\alpha_k}(t-\tau) \, \mathrm{d}W(\tau).
\end{equation}
The scalar resolvent function $s_\alpha(t)$ can be defined in terms of the well known Mittag-Leffler's function
$E_\rho(t)$ \cite{Mit1905AM};  this link will be exploited in order to prove finer properties of the function $s_\alpha(t)$, see Section~\ref{sec2}.
Notice further that the choice $\rho = 1$ reduces the problem to the evolution equation studied in \cite{IglTer1999EJP}.
In this case, $s_\alpha(t) = e^{-\alpha t}$ satisfies the semigroup property, $X_k$ is a Markov process and its behaviour at infinity is much easier to study than in the  case $1 < \rho \le 2$.

\subsection{Presentation of the results and outline of the paper}

In this paper, we assume that 
\begin{equation}\label{e1}
	\text{
		$\alpha_k$ are i.i.d., with a Gamma distribution $\Gamma(\mu, \lambda)$, $\mu, \lambda > 0$;
	}
\end{equation}
and we assume that the fractional integration parameter $\rho$ satisfies
\begin{equation*}
	\rho \in (1, 2].
\end{equation*}

Notice that condition \eqref{e1} is quite standard in the literature, compare, e.g., \cite{IglTer1999EJP}.
In our setting, we prove that the sequence of the \emph{empirical mean processes} $Y_n$ (given in \eqref{eq:emp}) converges
in $L^2(\bP)$, for fixed time,  $\bar\bP$-a.s., to a stochastic process $Y(t)$ given by the convolution Wiener integral
\begin{equation}
	\label{eq:4}
	Y(t) = \int_0^t \bar\bE[s_\alpha(t - s) ]  \, \mathrm{d}W(s),
\end{equation}
where
the $\bar{\bP}$ expectation $\bar{\bE}$ of the resolvent scalar kernel with respect to the random variable $\alpha$
can be explicitly computed and related to a generalized Wright-Fox function $G^\mu_\rho(z)$, see Section \ref{sec2.2} below.
In Section \ref{sez3}, we prove that $Y_n(t)$ converges pathwise on $[0,T]$, $\bar\bP$-a.s., for arbitrary $T > 0$.
First, in Theorem \ref{th:weak-conv}, we prove weak convergence of the laws in the space of continuous trajectories on $[0,T]$, by proving a suitable tightness result. Then, in Theorem \ref{th:path-conv}, we prove that the convergence holds almost surely, by using a particular application of the integration by parts formula.

Finally, in Section \ref{sez4} we consider the asymptotic behaviour of the limit process $Y(t)$ as $t \to \infty$.
Notice that in the previous section no condition was imposed on the values of the parameters $\mu$ and $\lambda$ defining the $\alpha$-distribution.
In this section, we shall impose the bound 
\begin{equation}\label{eq:condition9}
	\mu > \frac{1}{2\rho};
\end{equation}
notice that this condition is coherent with the condition $\mu > \frac12$ imposed in \cite{IglTer1999EJP} for the evolutionary problem ($\rho = 1$).

We show that the process $Y(t)$ converges in distribution, for $t \to \infty$, to some Gaussian random variable $\eta$ and we give a characterization of the limit as ``stationary'' solution of the problem.
The emphasis here comes from the fact that we cannot use the notion of invariant measures for SDEs, since our system is not Markovian.
However, if we consider the solution of a modified problem with initial time equal to $-\infty$, we have that this process has law equal to $\eta$ for all times.

\section{Some special functions}
\label{sec2}

\subsection{Mittag-Leffler's function}
We introduce the Mittag-Leffler's function $E_\rho(z)$ defined by the series
\begin{equation}
	\label{eq:5}
	E_\rho(z) = \sum_{k=0}^\infty  \frac{z^k}{\Gamma(\rho k + 1)}, \qquad
	z \in \bC,
\end{equation}
for any $\rho > 0$.
Notice that this is a generalization of the exponential function, obtained as a special case when $\rho = 1$.
This function was introduced by Mittag-Leffler in~\cite{Mit1905AM},
and since then studied by many authors.
Basic references for the next results are
\cite{ErdMagObeTri1955, GorKilMaiRog2020}.

The series in \eqref{eq:5} converges in the whole convex plane for $\rho > 0$.
For $x \in \bR$, $x \to \infty$, we get the following asymptotic expansion \cite[(3.4.15)]{GorKilMaiRog2020}: for any $m \ge 1$ we have
\begin{equation}\label{eq:ML-asypt-expans}
	E_\rho(-x) = - \sum_{k=1}^m \frac{(-x)^{-k}}{\Gamma[1 - k \rho]}  + O(|x|^{-m-1}).
\end{equation}

The case $0<\rho<1$ is somehow peculiar because the Mittag-Leffler's function $E_\rho(-x)$ is completely monotonic, that is for all $n \in \bN$, for all $x \in \bR_+$, $(-1)^{n}E_\rho^{(n)}(-x)\geq 0$.
It follows in particular that $E_\rho(-x)$ has no real zeroes for $x \in \bR_+$ when $0 < \rho < 1$.
Conversely, in the regime $1 < \rho < 2$, the Mittag-Leffler's function $E_\rho(-x)$ shows a damped oscillation converging to zero as
$x \to \infty$, as can be seen in Figure~\ref{fig:mittagleffler}.

\begin{figure}[ht]
	\begin{center}
		\includegraphics[width=.66 \textwidth]{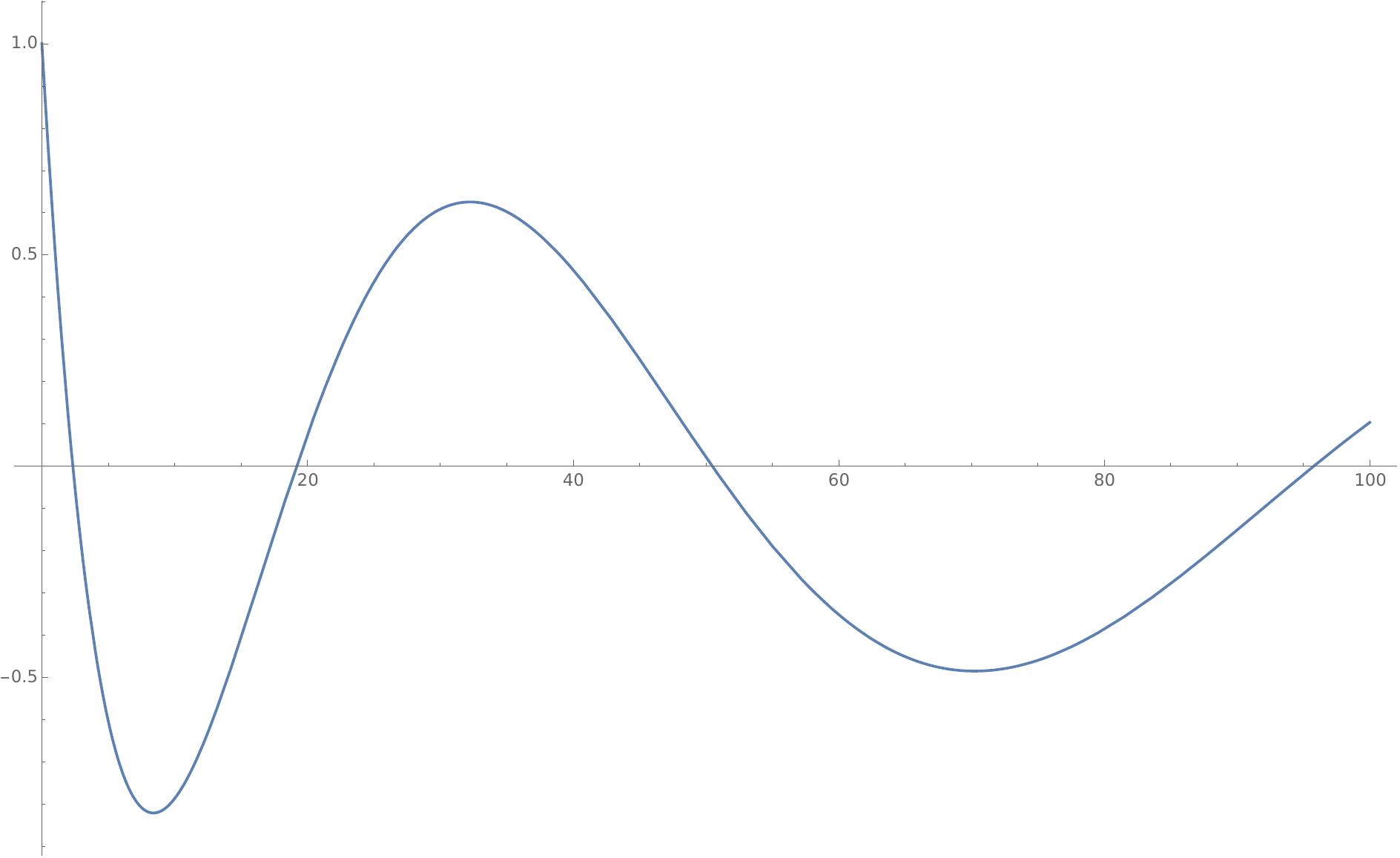}
	\end{center}
	\caption{The graph of the Mittag-Leffler's function $E_\rho(-x)$ with parameter $\rho = 1.9$. Plotted with Mathematica~\cite{Mathematica}.}
	\label{fig:mittagleffler}
\end{figure}

The global behaviour of the Mittag-Leffler's function $E_\rho(-x)$ is established in \cite[Corollary 3.1]{GorKilMaiRog2020}:
\begin{equation}
	\label{eq:ML-bdd}
	\forall \, x > 0: \quad |E_\rho(-x)| \le \frac{M}{1 + x},
\end{equation}
for some constant $M$ depending only on $\rho$.

\begin{lemma}
	For any $\rho > 0$, the derivative of the Mittag-Leffler's function $E_{\rho}(-x)$ satisfies
	\begin{equation}\label{eq:ML-deriv}
		\frac{d}{dx}E_\rho(-x) = -\frac1\rho E_{\rho,\rho}(-x),
	\end{equation}
	where $E_{\rho,\rho}(z)$ is the two-parametric Mittag-Leffler's function.
\end{lemma}

Recall that, by definition, the two-parametric Mittag-Leffler function is (see for example~\cite{GorKilMaiRog2020}),
\begin{equation*}
	E_{\rho,\tau}(z) = \sum_{k=0}^\infty \frac{z^k}{\Gamma(\rho k+\tau)}, \qquad z \in \bC, \Re(\rho)>0, \tau\in\bC.
\end{equation*}
As for the original Mittag-Leffler's function $E_\rho(z)$, also $E_{\rho,\rho}(z)$ is an entire function of order $\omega = \frac{1}{\rho}$.
Moreover, it is proved in \cite[(1.2.5)]{PopSed2013JMS} that a similar asymptotic expansion to the one given in \eqref{eq:ML-asypt-expans} for the Mittag-Leffler's function holds for the two-parametric Mittag-Leffler's function:
\begin{equation*}
	E_{\rho,\rho}(-x) = - \sum_{k=1}^m \frac{(-x)^{-k}}{\Gamma[\rho - k \rho]}  + O(|x|^{-m-1}).
\end{equation*}
However, since $\frac{1}{\Gamma[-n]} = 0$ for any integer $n \ge 1$, we see that the previous estimate becomes, for any $m \ge 1$,
\begin{equation*}
	E_{\rho,\rho}(x) = x^{-1}  + O(|x|^{-m-1}).
\end{equation*}
We are thus led to the following global estimate for the behaviour of the two-parametric Mittag-Leffler's function $E_{\rho,\rho}(-x)$:
\begin{equation}
	\label{eq:ML2-bdd}
	\forall \, x > 0: \quad |E_{\rho,\rho}(-x)| \le \frac{M_2}{1 + x},
\end{equation}
for some constant $M_2$ depending only on $\rho$.

\subsection{A generalized Wright-Fox type function} \label{sec2.2}
In this paper, we will encounter the following special function,
\begin{equation*}
	G_\rho^\mu(z) = \sum_{k=0}^\infty a_k z^k = \sum_{k=0}^\infty \frac{\left(\mu\right)_k}{\Gamma(k \rho + 1)} z^k,
\end{equation*}
where $(\mu)_n = \frac{\Gamma(\mu + n)}{\Gamma(\mu)}$ is the Pochhammer symbol. This function has some similarities to the Mittag-Leffler function, and in fact both are special cases of the generalized Wright-Fox type function, defined as
\[
\prescript{}{p}{\Psi}_q\left[\begin{matrix}
	(a_1,A_1),\ldots,(a_p,A_p)\\
	(b_1,B_1),\ldots,(b_q,B_q)
\end{matrix};z\right]=\sum_{k=0}^{\infty}\frac{\prod_{i=1}^p \Gamma(a_i+A_i k)}{\prod_{j=1}^q \Gamma(b_j+B_j k)}  \frac{z^k}{k!},
\]
where the series converges, with $a_i,b_j\in \bC$ and $A_i, B_j \in \bR$, see~\cite{Luc2019BT}. 
We recover both the two-parametric Mittag-Leffler function $E_{\rho, \tau}(z)$ and $G_\rho^\mu(z)$ in $\prescript{}{p}{\Psi}_q$ with the following choices:
\[
E_{\rho, \tau}(z) = \prescript{}{1}{\Psi_1}\left[\begin{matrix}
	(1,1)\\
	(\tau,\rho)
\end{matrix};z\right] \qquad G_\rho^\mu(z) = \prescript{}{2}{\Psi_1}\left[\begin{matrix}
	(1,1),(\mu,1)\\
	(1,\rho)
\end{matrix};z\right]
\]

In the next lemma, we prove that $G_\rho^\mu(z)$ is well defined and that it is
an entire function, as was $E_\rho$.

\begin{lemma}
	Assume that
	$\rho > 1$.
	Then the function $G_{\rho}(z)$ is an entire function of the variable $z \in \bC$, of order $\omega = \frac{1}{\rho - 1}$.
\end{lemma}

\begin{proof}
	Since $G_\rho^\mu$ is defined in terms of a power series,
	to prove that it is an entire function
	it is sufficient to study the radius of convergence.
	By the ratio test, it is sufficient to study the behaviour of
	\begin{align*}
		R & = \lim_{n \to \infty} \frac{a_n}{a_{n+1}} \\
		& = \lim_{n \to \infty} \frac{ \Gamma(\mu) \Gamma(\rho n + \rho + 1)}{ \Gamma(\mu + n + 1) } \frac{ \Gamma(\mu + n) } { \Gamma(\mu) \Gamma(\rho n + 1)}=  \frac{1}{\mu + n }\frac{\Gamma(\rho n + \rho + 1)}{\Gamma(\rho n + 1)}.
	\end{align*}
	Stirling's series can be used to find the asymptotic behaviour of the last ratio~\cite{TriErd1951PJM}:
	\begin{equation*}
		\frac{\Gamma(\rho n + 1)}{\Gamma(\rho n + \rho + 1)} = (\rho n)^{-\rho} \left(1 + O(n^{-1}) \right),
	\end{equation*}
	therefore, by taking the limit
	\begin{equation*}
		R = \lim_{n \to \infty} \frac{a_n}{a_{n+1}} = \lim_{n \to \infty} \frac{n^{\rho}}{\mu + n},
	\end{equation*}
	we see that $R$ is infinite provided $\rho > 1$.
	
	By using again Stirling's approximation, it is possible to compute the \emph{order} $\omega$ of the entire function $G_\rho^\mu(z)$.
	Recall that $\omega$ is the infimum of all $\kappa$ such that
	\begin{equation*}
		\max_{|z|=r} |G_\rho^\mu(z)| < e^{r^\kappa},
	\end{equation*}
	eventually for $r \to \infty$.
	For the entire function $G_\rho^\mu(z) = \sum a_k z^k$, the order is given by the formula
	\begin{equation*}
		\omega = \limsup_{k \to \infty} \frac{k \log(k)}{\log(1/|a_k|)};
	\end{equation*}
	for large $k$, we have
	\begin{align*}
		\frac{1}{a_k} \approx& \frac{e^{-k \rho} (k \rho)^{k \rho + 1/2}}{e^{-k} (k)^{k + \frac12 + \mu}},
		\\
		\log(1/a_k) \approx& (\rho - 1) k \log(k) + (1 - \rho) k + (k\rho + 1/2) \log(\rho) - \mu \log(k),
	\end{align*}
	which implies
	$\omega = \frac{1}{\rho-1}$.
\end{proof}

\begin{remark}
	Further properties of the function $G_\rho^\mu$ will be proved in the sequel.
	In particular, as a corollary of the computation in Lemma \ref{le:Cauchy-L2},
	we obtain the following asymptotic behaviour for large $x$:
	\begin{equation}
		\label{eq:G-is-square-int}
		\left| G_\rho^\mu(-x) \right| \le C \left(\frac1x + \frac1{x^\mu}\right), \qquad x \to \infty,
	\end{equation}
	for $\mu \not= 1$.
	A plot of he restriction of $G_\rho^\mu$ to the negative real numbers can be seen in Figure~\ref{fig:generalized_ML}.
	\begin{figure}[ht]\begin{center}
			\includegraphics[width=.66 \textwidth]{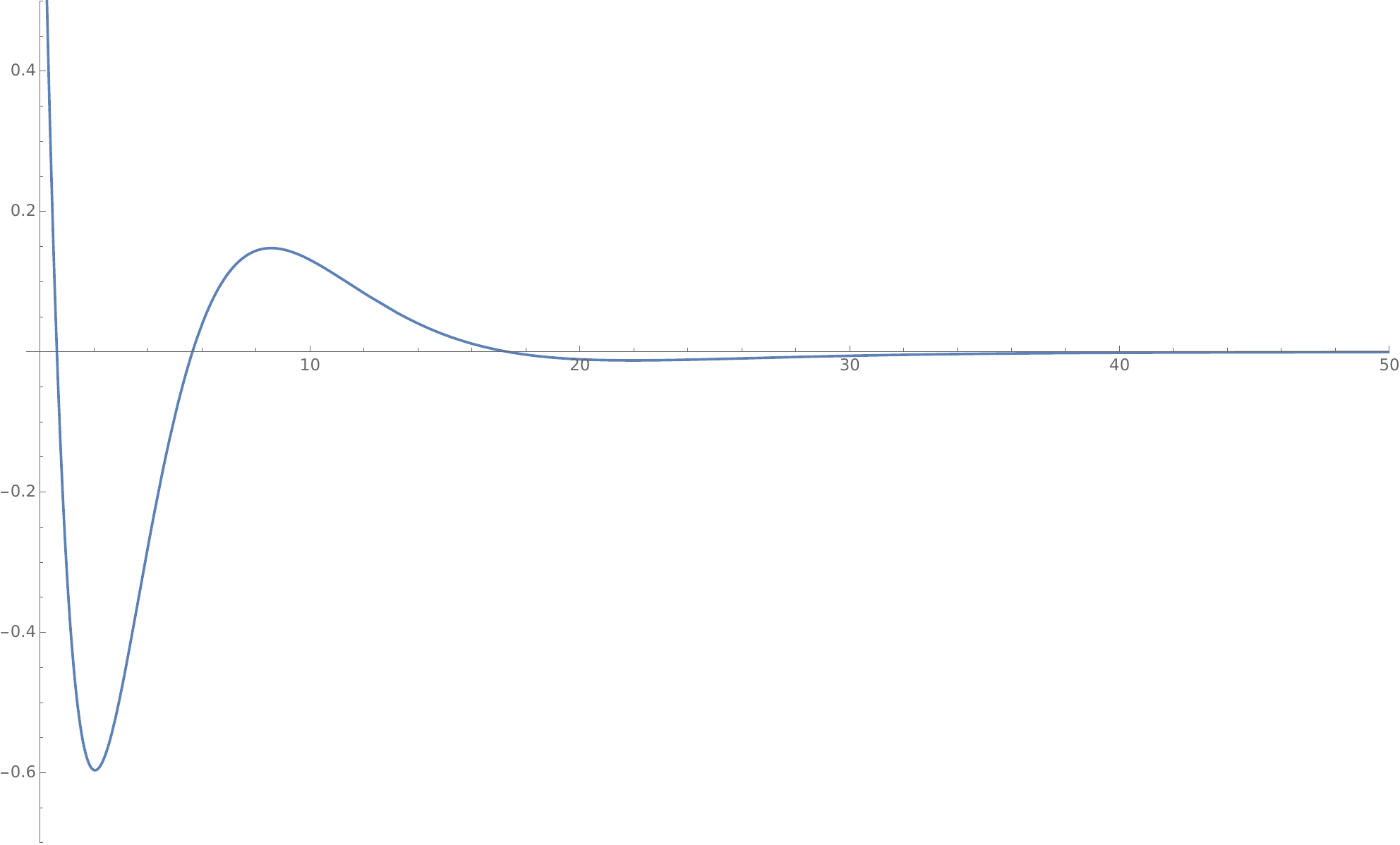}
		\end{center}
		\caption{The graph of the function $G_\rho^\mu(-x)$ with parameters $\mu=4$ and $\rho = 1.9$. Plotted with Mathematica~\cite{Mathematica}.}
		\label{fig:generalized_ML}
	\end{figure}
\end{remark}

\section{The empirical mean process}
\label{sez3}

We consider in this section the solution of the fractional evolution equation~\eqref{eq:1}
and the properties of the empirical mean process $Y_n(t)$ defined in~\eqref{eq:emp}.

The relation between Mittag-Leffler's function and the fractional derivative of Riemann-Liouville type is specified through the following relation \cite[(3.7.44)]{GorKilMaiRog2020}
\begin{equation*}
	E_\rho(\lambda t^\rho) - {\lambda} \int_0^t E_\rho(\lambda s^\rho) g_\rho(t-s) \, \mathrm{d}s = 1.
\end{equation*}
By comparing it with the definition of the scalar resolvent function $s_\alpha(t)$ in~\eqref{eq:2} we obtain
the relation
\begin{equation}
	\label{eq:res-ML}
	s_\alpha(t) = E_\rho(- \alpha t^\rho), \qquad t > 0, \text{ for all $\alpha > 0$.}
\end{equation}

For the sake of completeness, let us recall the behaviour of the Gamma distribution.

\begin{proposition}
	On the probability space $(\bar \Omega, \bar \cF, \bar \bP)$,
	assume that  $\alpha$ is a random variable with a Gamma distribution of parameters $\mu$ and $\lambda$, i.e.,
	$\alpha$ has the density
	\begin{equation*}
		f(x \mid \mu, \lambda) = \frac{1}{\Gamma(\mu)} \lambda e^{-\lambda x} \left(\lambda x\right)^{\mu - 1}
		\qquad x > 0;
	\end{equation*}
	then the moments of $\alpha$ are given by
	\begin{equation}\label{eq:7}
		\bar{\bE}[\alpha^n] = \frac{1}{\lambda^n} \left(\mu\right)_n.
	\end{equation}
\end{proposition}

We apply the previous proposition to compute the expected value (with respect to the measure $\bar\bP$) of the random process $s_\alpha(t)$.

\begin{lemma}
	On the probability space $(\bar \Omega, \bar \cF, \bar \bP)$,
	assume that  $\alpha$ has a Gamma distribution of parameters $\mu$ and $\lambda$.
	Then the expected value of $s_\alpha(t)$
	satisfies
	\begin{equation*}
		\bar \bE[ s_\alpha(t)] = G_{\rho}^\mu(- t^\rho/\lambda) =
		\sum_{k=0}^\infty  \frac{\left(\mu\right)_k}{\Gamma(k\rho+1)}   \left( -\frac{t^{\rho}}{\lambda} \right)^k.
	\end{equation*}
\end{lemma}

\begin{proof}
	Let us consider the representation of the scalar resolvent kernel $s_\alpha(t)$ in terms of the Mittag-Leffler's function $E_\rho(x)$ given in \eqref{eq:res-ML} and the definition of the Mittag-Leffler's function as an entire function given in~\eqref{eq:5}:
	\begin{equation*}
		\bar\bE[ s_\alpha(t) ] =  \bar\bE \left[ \sum_{k=0}^\infty  \frac{(-1)^k}{\Gamma(k\rho+1)} \alpha^k t^{k\rho} \right].
	\end{equation*}
	By an application of Fubini's theorem we obtain
	\begin{equation*}
		\bar\bE[ s_\alpha(t) ] =   \sum_{k=0}^\infty  \frac{(-1)^k}{\Gamma(k\rho+1)}\bar\bE \left[ \alpha^k \right] t^{k\rho},
	\end{equation*}
	and by \eqref{eq:7} we obtain
	\begin{equation*}
		\bar\bE[ s_\alpha(t) ] =   \sum_{k=0}^\infty  \frac{(-1)^k}{\Gamma(k\rho+1)}  \left(\mu\right)_k \left( \frac{t^{\rho}}{\lambda} \right)^k. \qedhere
	\end{equation*}
\end{proof}

\subsection{The limit of the empirical mean process}

In this section, we will study the asymptotic behaviour of the empirical mean process
\begin{equation*}
	Y_n(t) = \int_{0}^t \frac{1}{n} \sum_{k=1}^n s_{\alpha_k}(t - \tau) \, \mathrm{d}W(\tau).
\end{equation*}
By the law of the large numbers
\begin{equation*}
	\frac{1}{n} \sum_{k=1}^n s_{\alpha_k}(t) \longrightarrow \bar\bE[s_\alpha(t)], \qquad t > 0,
\end{equation*}
hence a natural candidate for the limit process of $Y_n(t)$ is given by
\begin{equation*}
	Y(t) = \int_0^t G_{\rho}^\mu\left(-\tfrac{(t-s)^\rho}{\lambda}\right) \, \mathrm{d}W(s).
\end{equation*}

\begin{theorem}\label{thm:l2-conv-Yn}
	The sequence $Y_n(t)$  converges in $L^2(\Omega, \bP)$, uniformly in $[0,T]$,
	to $Y(t)$, $\bar\bP$-a.s.
\end{theorem}

\begin{proof}
	In order to prove the convergence of $Y_n$ to $Y$, we first notice that
	\begin{equation*}
		f_n(t) = \frac{1}{n} \sum_{k=1}^n s_{\alpha_k}(t)
	\end{equation*}
	satisfies, by estimate \eqref{eq:ML-bdd},
	\begin{equation*}
		|f_n(t)| \le M;
	\end{equation*}
	therefore, the same applies to the function
	$G_{\rho}^\mu(- t^\rho/\lambda)$, and $Y(t)$ is a Gaussian process.
	
	Moreover,
	the function $s_\alpha(t)$ is continuous and bounded for $\alpha \in \bR_+$ and $t \in [0,T]$, for arbitrary $T > 0$, then the uniform law of large numbers  \cite[Lemma 2.4]{NewMcF1994HoE} implies that $G_{\rho}^\mu(- t^\rho/\lambda)$ is continuous and bounded in $t \in [0,T]$ and
	\begin{equation*}
		\sup_{t \in [0,T]} \left| \frac{1}{n} \sum_{k=1}^n s_{\alpha_k}(t) - \bar\bE[s_\alpha(t)] \right| \to 0 \quad \bar\bP-a.s.
	\end{equation*}
	It follows from the dominated convergence theorem that
	\begin{equation*}
		\bE \sup_{t \in [0,T]} |Y_n(t) - Y(t)|^2 \le 4 \int_0^T \left| \frac{1}{n} \sum_{k=1}^n s_{\alpha_k}(t) - \bar\bE[s_\alpha(t)] \right|^2 \, \mathrm{d}t \to 0 \quad \bar\bP-a.s. \qedhere
	\end{equation*}
\end{proof}

\subsection{Weak convergence of the empirical mean process in \texorpdfstring{$C([0,T])$}{C[0,T]}}

We have proved before that $Y_n(t)$ and $Y(t)$ are, $\bar\bP$-a.s., Gaussian, continuous stochastic processes on $(\Omega, \cF, \bP)$.

In the remaining of this section, the whole construction is thought to hold $\bar\bP$-a.s., meaning that we consider $\alpha$ given and we are in the space $(\Omega, \cF, \bP)$.
We aim to consider the
$\bP$-weak convergence in the space $C([0,T])$ of $Y_n$ to $Y$:
\begin{multline*}
	Y_n \Rightarrow Y \text{ if and only if $\phi(Y_n) \to \phi(Y)$ in distribution, } \\  \text{for every bounded Borel measurable $\phi: C([0,T]) \to \bR$.}
\end{multline*}

\begin{theorem}
	\label{th:weak-conv}
	The sequence $Y_n$ $\bP$-weakly converges to the process $Y$ in the space $C([0,T])$, $\bar\bP$-almost surely.
\end{theorem}

By \cite[Theorems 12.3 and 8.1]{Bil1968},  weak convergence of $Y_n$ to $Y$ follows if we prove that the following two conditions hold
\begin{enumerate}
	\item convergence of the finite dimensional distributions, and
	\item the moment condition:
	\begin{equation}
		\label{eq:tight}
		\bE[|Y_n(t) - Y_n(s)|^\kappa] \le C(\kappa, \delta) |t-s|^{1+\delta},
	\end{equation}
	where $C(\kappa, \delta)$ depends on the constants $\kappa$ and $\delta$ only,
	compare \cite[formula (12.51)]{Bil1968}.
\end{enumerate}

We shall prove each of the above-described conditions in the following two lemmas, which therefore together imply the thesis of the theorem.

\begin{lemma}
	\label{le:finite-distrib}
	For any $k \in \bN$, $t_1, \dots, t_k \in [0,T]$, $x_1, \dots, x_k \in \bR$,
	it holds that  $\bP(Y_n(t_1) \le x_1, \dots, Y_n(t_k) \le x_k)$ converges to
	$\bP(Y(t_1) \le x_1, \dots, Y(t_k) \le x_k)$ as $n\to \infty$.
\end{lemma}

\begin{proof}
	Both $Y_n$ and $Y$ are Gaussian processes, so their finite dimensional distributions are fully determined by their mean and covariance. For any $k \in \bN$, $t_1, \dots, t_k \in [0,T]$, the mean and covariance of $Y_n(t_1),\ldots, Y_n(t_k)$ converge to mean and covariance of $Y(t_1),\ldots,Y(t_k)$, thanks to Theorem~\ref{thm:l2-conv-Yn} and the pointwise (in $t$) convergence of $Y_n(t) \to Y(t)$. 
\end{proof}

\begin{lemma}
	\label{le:tight}
	The moment condition \eqref{eq:tight} holds, and the family of empirical mean processes $\{Y_n\}_n$ is tight.
\end{lemma}

\begin{proof}
	First, we notice that by exploiting the Gaussianity of $Y_n$, it is sufficient to search an estimate for $\kappa = 2$. Further, we can reduce the problem to the estimate of a single $X_k$, since
	\begin{equation*}
		\bE[|Y_n(t) - Y_n(s)|^2] \le \frac1n \sum_{k=1}^n \bE[|X_k(t) - X_k(s)|^2].
	\end{equation*}
	Next, we write (for simplicity, we drop the index $k$ in the following computation)
	\begin{align*}
		\bE[|X(t) -& X(s)|^2] \\& = \bE \left| \int_0^s [s_\alpha(t-u) - s_\alpha(s-u)] \, \mathrm{d}W(u) + \int_s^t s_\alpha(t-u) \, \mathrm{d}W(u) \right|^2
		\\
		& \le 2 \int_0^s [s_\alpha(t-u) - s_\alpha(s-u)]^2 \, \mathrm{d}u + 2 \int_s^t [s_\alpha(t-u)]^2 \, \mathrm{d}u
		\\
		& \le 2 \int_0^s [s_\alpha(t-s+u) - s_\alpha(u)]^2 \, \mathrm{d}u + 2 \int_0^{t-s} [s_\alpha(u)]^2 \, \mathrm{d}u.
	\end{align*}
	Recall that, by estimate \eqref{eq:ML-bdd},
	$	|s_\alpha(t)| = |E_\rho(-\alpha t^\rho)| \le M$,
	hence the second integral in the right hand side of previous estimate is bounded by a constant $2M^2$ times $(t-s)$.
	As far as the first term is concerned, we get
	\begin{align*}
		\int_0^s [s_\alpha(t-&s+u) - s_\alpha(u)]^2 \, \mathrm{d}u
		\\
		& =
		\int_0^s [E_\rho(-\alpha (t-s+u)^\rho) - E_\rho(-\alpha u^\rho)]^2 \, \mathrm{d}u
		\\
		& \stackrel{\mathclap{[v = \alpha^{1/\rho} u]}}{=}
		\qquad\alpha^{-1/\rho} \int_0^{\alpha^{1/\rho} s} [E_\rho(-(\alpha^{1/\rho} (t-s) + v)^\rho) - E_\rho(-v^\rho)]^2 \, \mathrm{d}v
		\\
		& = \alpha^{-1/\rho} \int_0^{\alpha^{1/\rho} s} [s_1(\alpha^{1/\rho} (t-s) + v) - s_1(v)]^2 \, \mathrm{d}v
		\\
		& = \alpha^{-1/\rho} \int_0^{\alpha^{1/\rho} s}  \left| \int_v^{v + \alpha^{1/\rho} (t-s)} s_1'(x) \, \mathrm{d}x \right|^2 \, \mathrm{d}v.
	\end{align*}
	The derivative of the scalar resolvent kernel can be written in terms of the two-parametric Mittag-Leffler's function, see formula \eqref{eq:ML-deriv}
	\begin{equation*}
		s_1'(x) = \frac{d}{dx} E_\rho(- x^\rho) = - x^{\rho - 1}  E_{\rho,\rho}(-x^\rho).
	\end{equation*}
	Recalling the bound \eqref{eq:ML2-bdd} we see that
	\begin{equation*}
		\left| s'_1(x) \right| \le x^{\rho - 1} \frac{M_2}{1 + x^\rho} \le \frac{M_3}{1 + x},
	\end{equation*}
	for some constant $M_3$.
	\\
	In conclusion, we get
	\begin{equation*}
		\bE[|X(t) - X(s)|^2]
		\le [2 M_3^2 \alpha^{2/\rho} T^2 + 2 M^2] (t-s).
	\end{equation*}
	This implies
	\begin{equation*}
		\bE[|Y_n(t) - Y_n(s)|^2] \le K_n (t-s),
	\end{equation*}
	where the constant $K_n$ can be explicitly given as
	\begin{equation*}
		K_n = 2 M_3^2 T^2 \left( \frac1n \sum_{k=1}^n (\alpha_k)^{2/\rho} \right) + 2 M^2.
	\end{equation*}
	Notice that, $\bar\bP$-almost surely,
	\begin{equation*}
		K_n  \xrightarrow[n \to \infty]{} \bar K = 2 M_3^2 T^2 \bar\bE[ \alpha^{2/\rho} ] + 2 M^2.
	\end{equation*}
	It follows that for some $n_0 \in \bN$, $K_n \le \bar K + 1$ for any $n \ge n_0$. Set
	\begin{equation*}
		C(2,0) = \max_{n \in \{1, \dots, n_0\}} K_n \vee (\bar K + 1).
	\end{equation*}
	Since $Y_n(t) - Y_n(s)$ is a centred Gaussian random variable, by standard arguments we conclude that
	$\bE[|Y_n(t) - Y_n(s)|^{2n}] \le C(2n, n-1) |t-s|^{n}$ and
	$C(2n, n-1) = (2n-1)!! \, C(2,0)^n$.
\end{proof}

\subsection{Almost sure convergence of the empirical mean process in \texorpdfstring{$C([0,T])$}{C[0,T]}}

In this section, we consider the convergence of the paths of $Y_n$ to $Y$ in $C([0,T])$ with respect to the uniform topology.

\begin{theorem}
	\label{th:path-conv}
	In our assumptions, we have that
	\begin{equation*}
		\lim_{n \to \infty} \sup_{t \in [0,T]} |Y_n(t) - Y(t)| \to 0,
	\end{equation*}
	almost surely with respect to $\bP \times \bar\bP$.
\end{theorem}

Our proof will follow from a somehow broader result.
Let $f_n$ be a sequence of smooth functions on $[0,T]$ that converges
(at least pointwise) to~$f$. We shall see later which further assumptions are required.
Then we compute
\begin{align*}
	(f_n * & \dot W)(t) - (f*\dot W)(t)
	\\
	& =
	f_n(t-s)W(s) \big|_{s=0}^{s=t} + (W * \dot f_n)(t)
	- f(t-s)W(s) \big|_{s=0}^{s=t} - (W * \dot f)(t)
	\\
	& = [f_n(0) - f(0)] W(t) + (W * \dot{(f_n - f)})(t).
\end{align*}
It follows that  there exists a constant $C$ finite $\bP$-a.s.\ such that
\begin{multline*}
	\sup_{t \in [0,T]} \left| (f_n * \dot W)(t) - (f*\dot W)(t)  \right|
	\\ \le C \left( |f_n(0) - f(0)| + \int_0^T \left| \dot{f_n}(s) - \dot{f}(s) \right| \, \mathrm{d}s \right).
\end{multline*}
Therefore, a set of sufficient conditions for the required convergence is given by
\begin{enumerate}
	\item $f_n(0) \to f(0)$,
	\item $\dot {f_n}(s) \to \dot{f}(s)$ pointwise in $[0,T]$, and
	\item $\displaystyle \sup_{n \in \bN} \sup_{s \in [0,T]} |\dot{f_n}(s)| \le C$.
\end{enumerate}
Then an application of Lebesgue's dominated convergence theorem shows that the following lemma holds.

\begin{lemma}
	\label{le:conv}
	Under conditions 1.~to 3., the stochastic convolution processes $X_n = (f_n * \dot W)$ converge in $C([0,T])$ (endowed with the uniform topology) to the process $X = (f * \dot W)$.
\end{lemma}

We are now ready to prove the main result of this section.

\begin{proof}[Proof of Theorem \ref{th:path-conv}]
	With the notation
	\begin{equation*}
		f_n(t) = \frac{1}{n} \sum_{k=1}^n s_{\alpha_k}(t), \qquad f(t) = G_\rho^\mu(-t^{\rho}/\lambda),
	\end{equation*}
	introduced above, we see that
	\begin{equation*}
		\dot{f_n}(t) =  \frac{1}{n} \sum_{k=1}^n \frac{\mathrm{d}}{\mathrm{d}t} s_{\alpha_k}(t),
	\end{equation*}
	and by using the identity $s_\alpha(t) = E_\rho(-\alpha t^{\rho})$
	and formula \eqref{eq:ML-deriv}
	we compute
	\begin{equation*}
		\tfrac{d}{dt} E_\rho(-\alpha t^\rho) = - \alpha t^{\rho - 1} E_{\rho,\rho}(-\alpha t^\rho),
	\end{equation*}
	where $E_{\rho,\rho}(x)$ is the two-parametric Mittag-Leffler's function;
	recalling the bound
	\eqref{eq:ML2-bdd}, we get
	\begin{equation*}
		\left| \tfrac{d}{dt} E_\rho(-\alpha t^\rho) \right| \le \alpha t^{\rho - 1} \frac{M_2}{1 + \alpha t^{\rho}},
	\end{equation*}
	which implies the bound
	\begin{equation*}
		\left| \tfrac{d}{dt} E_\rho(-\alpha t^\rho) \right| \le
		M \alpha^{1/\rho} \underbrace{\sup_{x > 0} \frac{x^{\rho-1}}{1+x^\rho}}_{\le 1}.
	\end{equation*}
	Since $\alpha$ has a Gamma distribution of parameters $\mu$ and $\lambda$, we compute
	\begin{align*}
		\bar\bE[\alpha^{1/\rho}] & = \frac{1}{\Gamma(\mu)} \int_0^\infty x^{1/\rho} \lambda e^{-\lambda x} \left(\lambda x\right)^{\mu-1} \, \mathrm{d}x
		\\
		& = \frac{1}{\lambda^{1/\rho} \Gamma(\mu)} \int_0^\infty \lambda e^{-\lambda x} \left(\lambda x\right)^{\mu+1/\rho-1} \, \mathrm{d}x
		= \frac{\Gamma(\mu + 1/\rho)}{\lambda^{1/\rho} \Gamma(\mu)}.
	\end{align*}
	
	Therefore, we get
	\begin{equation*}
		\sup_{t \in [0,T]} \left| \dot{f_n}(t) \right|
		\le M \, \frac1n \sum_{k=1}^n \left(\alpha_k\right)^{1/\rho};
	\end{equation*}
	by the law of large numbers in the space $(\bar\Omega, \bar\bP)$
	we obtain that
	\begin{equation*}
		\frac1n \sum_{k=1}^n \left(\alpha_k\right)^{1/\rho} \to \bar\bE[\alpha^{1/\rho}] < \infty,
	\end{equation*}
	hence there exists a constant
	$C = C(\bar\omega)$, finite $\bar\bP$-almost surely, such that
	\begin{equation*}
		\sup_{n \in \bN} \sup_{t \in [0,T]} \left| \dot{f_n}(t) \right|
		\le C.
	\end{equation*}
	
	By another application of the law of large numbers we get that the sequence
	\begin{equation*}
		\dot{f_n}(t) =  \frac{1}{n} \sum_{k=1}^n \frac{\mathrm{d}}{\mathrm{d}t} s_{\alpha_k}(t)
	\end{equation*}
	converges to
	$
	G_{\rho,\rho}^\mu(t) = \bar\bE \left[ -\alpha t^{\rho-1} E_{\rho,\rho}(-\alpha t^{\rho}) \right]$.
	Since
	$	G_\rho^\mu(t^\rho/\lambda) = \bar\bE \left[ E_\rho(-\alpha t^\rho) \right]$,
	taking the derivative inside the expectation  we
	obtain that
	\begin{equation*}
		\frac{\mathrm{d}}{\mathrm{d}t} G_\rho^\mu(t^\rho/\lambda) = G_{\rho,\rho}^\mu(t),
	\end{equation*}
	that is the pointwise convergence required in condition 2.
	\\
	Finally, since $f_n(0) = 1 = f(0)$ for every $n \in \bN$, we conclude by applying Lemma~\ref{le:conv}. 
\end{proof}

\section{Asymptotic behaviour and stationary solutions}
\label{sez4}

In this section we discuss the asymptotic behaviour as $t \to \infty$ of the stochastic process $Y(t)$ defined in \eqref{eq:4}.
We already know that $Y(t)$ is a Gaussian process with zero mean and variance given by
\begin{equation*}
	\sigma^2_t = \int_0^t \left| G_\rho^\mu(-u^\rho/\lambda)\right|^2 \, \mathrm{d}u = \int_0^t  \bar\bE[s_\alpha(u)]^2 \, \mathrm{d}u.
\end{equation*}
We are interested in the convergence of the laws $\cL(Y(t))$ as $t \to \infty$.

A standard assumption in this section will be the following:
\begin{equation}\label{eq:con-mu-rho}
	\mu > \frac{1}{2\rho}.
\end{equation}
\begin{remark}\label{rmk:condition}
	By a direct computation, we see that the random variable $\alpha^{-1/\rho}$ belongs to $L^2(\bar\Omega, \bar\bP)$ if and only if $\mu - 2\rho > 0$, {\sl i.e.}~condition \eqref{eq:con-mu-rho} holds.
\end{remark}

\begin{theorem}
	\label{th:converg}
	Assume that $\alpha$ has a Gamma distribution with parameters
	$\mu > \frac{1}{2\rho}$
	and $\lambda > 0$.
	There exists a Gaussian random variable $\eta$ such that $Y(t)$ converges in law to $\eta$ as $t \to \infty$.
\end{theorem}

We first extend the Wiener process $W(t)$ to a process with time $t \in \bR$. Let $\tilde W(t)$ be a standard Brownian motion on the space $(\Omega, \cF, \bP)$, independent of $W(t)$, and set for $t\ge 0$ 
$W(-t) = \tilde W(t)$. 

Now, for any $-s < 0$, define the stochastic convolution process
\begin{equation*}
	Y_{-s}(t) = \int_{-s}^t  \bar\bE[s_\alpha(t-u)] \, \mathrm{d}W(u).
\end{equation*}
It is easily seen that
$\cL( Y_{-s}(0) ) = \cL(Y(s))$, for $s \ge 0$.

In next lemma we prove that $\{Y_{-s}(0)\}$ forms a Cauchy sequence in the space $L^2(\Omega, \bP)$. To do so, we rely on 

\begin{lemma}
	\label{le:Cauchy-L2}
	Under our assumptions, for any $\varepsilon > 0$ there exists $T > 0$ such that, for any $t > s > T$, 
	$	\bE \left| Y_{-t}(0) - Y_{-s}(0) \right|^2 \le \varepsilon$.
\end{lemma}

\begin{proof}
	Recall that
	\begin{equation*}
		Y_{-t}(0) = \int_{-t}^0 \bar\bE[s_\alpha(-u)] \, \mathrm{d}W(u)
	\end{equation*}
	is a centred Gaussian process, where $s_\alpha(t) = E_{\rho}(-\alpha t^{\rho})$ satisfies the bound
	\begin{equation*}
		|s_\alpha(t)| \le \frac{M}{1 + \alpha t^\rho}.
	\end{equation*}
	In order to prove the mean square convergence of $Y_{-t}(0)$ as $t \to \infty$, we need to prove that the quantity
	$	\bE[ \left| Y_{-t}(0) - Y_{-s}(0) \right|^2 ]$
	is bounded by a quantity that goes to 0 in $T$, for $t > s > T$.
	
	We estimate
	\begin{multline*}
		\bE[ \left| Y_{-t}(0) - Y_{-s}(0) \right|^2 ]\\
		= \int_{-s}^{-t} \left| \bar\bE[s_\alpha(-u)] \right|^2 \, \mathrm{d}u
		\int_s^t \left( \bar\bE \left|s_\alpha(u) \right| \right)^2 \, \mathrm{d}u
		\int_s^t \left( \bar\bE \left|\frac{M}{1 + \alpha u^\rho} \right| \right)^2 \, \mathrm{d}u.
	\end{multline*}
	Since $\alpha$ has a Gamma distribution with parameters $\mu$ and $\lambda$, we have
	\begin{align*}
		\bE[ | &Y_{-t}(0) - Y_{-s}(0) |^2 ] \\ & \le \int_s^t \left( \int_0^\infty \frac{M}{1 + y u^\rho} \lambda (\lambda y)^{\mu - 1} e^{-\lambda y} \, \mathrm{d}y \right)^2 \, \mathrm{d}u
		\\
		& \stackrel{\mathclap{[z = \lambda y]}}{=}\ \int_s^t \left( \frac{1}{\Gamma(\mu)} \int_0^\infty \frac{M}{1 + z \frac{u^\rho}{\lambda}}   z^{\mu - 1} e^{-z} \, \mathrm{d}z \right)^2 \, \mathrm{d}u
		\\
		& \le
		\frac{M^2}{\Gamma(\mu)^2} \int_s^t \left( \int_0^1 \frac{1}{1 + z \frac{u^\rho}{\lambda}} z^{\mu - 1} e^{-z} \, \mathrm{d}z + \int_1^\infty \frac{1}{1 + z \frac{u^\rho}{\lambda}} z^{\mu - 1} e^{-z} \, \mathrm{d}z \right)^2 \, \mathrm{d}u
		\\
		& \le
		\frac{M^2}{\Gamma(\mu)^2} \int_s^t \left( \int_0^1 \frac{1}{1 + z \frac{u^\rho}{\lambda}} z^{\mu - 1} \, \mathrm{d}z + \int_1^\infty \frac{1}{\frac{u^\rho}{\lambda}} z^{\mu - 1} e^{-z} \, \mathrm{d}z \right)^2 \, \mathrm{d}u
		\\
		& \le
		\frac{M^2}{\Gamma(\mu)^2} \int_s^t \left( F_{2,1}(1, \mu, 1+\mu, -\tfrac{u^\rho}{\lambda})
		+ \Gamma(\mu) \frac{\lambda}{u^\rho} \right)^2 \, \mathrm{d}u,
	\end{align*}
	where $F_{2,1}$ is the hypergeometric function; notice that, for large $u$, we have
	\begin{align*}
		F_{2,1}(1, \mu, 1+\mu, -\tfrac{u^\rho}{\lambda}) \sim
		\begin{cases}
			c_{\lambda,\mu} u^{-\mu \rho}, \qquad & 0 < \mu < 1 \\
			c_{\lambda,1} \frac{\log(u)}{u}, \qquad & \mu = 1 \\
			c_{\lambda,\mu} u^{-\rho}, \qquad & \mu > 1,
		\end{cases}
	\end{align*}
	hence
	\begin{multline*}
		\bE[ \left| Y_{-t}(0) - Y_{-s}(0) \right|^2 ] \\ \le C_{M,\mu,\lambda,\rho} \left(
		\int_s^t \left\{  u^{-2 \rho} \mathbbm{1}_{\mu > 1} + \frac{\log(u)^2}{u^2} \mathbbm{1}_{\mu=1} + u^{-2\mu \rho} \mathbbm{1}_{0< \mu < 1} \right\} \, \mathrm{d}u \right)
	\end{multline*}
	and the right hand side is bounded, for every $\mu > \frac{1}{2\rho}$, for every $t > s > T$, by
	$	C \, T^{1-2\rho(\mu \wedge 1)}$
	(for $\mu = 1$, the bound is given by $\frac{(2 + \log(T))^2}{T}$); under the assumption $\mu > \frac{1}{2\rho}$, the exponent is always negative, so the quantity
	$\bE[ \left| Y_{-t}(0) - Y_{-s}(0) \right|^2 ]$ can be taken arbitrarily small by choosing $T$ large enough. 
\end{proof}

\begin{proof}[Proof of Theorem \ref{th:converg}]
	The result follows from the previous lemma.
	We have seen that $\{Y_{-t}(0)\}$ is a Cauchy sequence, hence it converges in $L^2(\Omega, \bP)$ to some random variable~$\eta$.
	
	Now, a simple modification of the computations leading to Lemma \ref{le:Cauchy-L2} implies that
	\begin{equation*}
		\sup_{t > 0} \bE\left[ |Y_{-t}(0)|^2 \right] \le C < \infty,
	\end{equation*}
	and the stability of the Gaussian distribution with respect to the convergence in law implies that $\eta$ has a Gaussian distribution as well, with variance
	\begin{equation*}
		\sigma^2 = \sup_{t > 0} \bE\left[ |Y_{-t}(0)|^2 \right] = \lim_{t \to \infty} \bE\left[ |Y_{-t}(0)|^2 \right].
	\end{equation*}
	But, as already stated above, $Y(t) \sim Y_{-t}(0)$, and we obtain the thesis. 
\end{proof}

Let us consider the following modification of our setting.
Let $W(t)$ be a two-side Wiener process, and define
\begin{equation*}
	\xi_k(t) = \int_{-\infty}^t s_{\alpha_k}(t-s) \, \mathrm{d}W(s),
\end{equation*}
where, as opposite to \eqref{eq:def-Xk}, we are taking the integral on an infinite time interval $(-\infty,t)$.

\begin{proposition}
	The process $\xi_k(t)$ is a stationary Gaussian process with 
	\[
	\cL(\xi_k(t)) \sim \cN(0, \|s_\alpha(\cdot) \|_{L^2(\bR_+)}^2),
	\]
	$\bar\bP$-almost surely.
\end{proposition}

\begin{proof}
	The only point worth some care is the fact that the variance is finite.
	This follows from the computation in \cite{BonTub2003SAaA}, see in particular Remark 3.2 therein, where it is proved that
	\begin{equation*}
		\|s_\alpha(\cdot)\|^2_{L^2(\bR_+)} \le c \frac{1}{\alpha^{1/\rho}}. \qedhere
	\end{equation*}
\end{proof}

\begin{theorem}
	Assume that condition \eqref{eq:con-mu-rho} holds.
	Then the process
	\begin{equation}\label{eq:def-eta-t}
		\eta(t) = \int_{-\infty}^t G_\rho^\mu(-(t-s)^\rho/\lambda) \, \mathrm{d}W(s)
	\end{equation}
	is a stationary Gaussian process with $\cL(\eta(t)) = \cL(\eta) \sim \cN(0, \sigma^2)$,
	where $\eta$ is the random variable defined in Theorem \ref{th:converg}.
\end{theorem}

\begin{proof}
	This follows from the computation leading to the proof of Lemma~\ref{le:Cauchy-L2}
	(compare also to Remark~\ref{rmk:condition} for the necessity of condition \eqref{eq:con-mu-rho}). 
\end{proof}

\begin{theorem}\label{thm:thm6}
	Consider the empirical mean processes $\eta_n(t)$, defined by
	\begin{equation*}
		\eta_n(t) = \frac1n \sum_{k=1}^n \xi_k(t), \qquad t \in \bR.
	\end{equation*}
	Then, for every $t \in \bR$, the sequence $\eta_n(t)$
	converges in $L^2(\Omega, \bP)$ as $n \to \infty$ to the random variable $\eta(t)$ defined in \eqref{eq:def-eta-t}, $\bar\bP$-almost surely.
\end{theorem}

Before we provide the proof, let us recall a few known results about convergence of sequence of Gaussian random variables (see, e.g.,~\cite{Bal2017}).
\begin{lemma}
	\label{le:Baldi}
	Assume that $\{X_n\}$ is a sequence of Gaussian random variables converging in law to a random variable $X$. Then $X$ is also a Gaussian random variable.
	
	Assume that $X_n \sim \mathcal{N}(m_n, \sigma^2_n)$ for every $n$ and that $m_n \to m$ and $\sigma^2_n \to \sigma^2$ as $n \to \infty$. Then $X_n$ converges in law to $X \sim \mathcal{N}(m, \sigma^2)$.
	
	Assume further that $\sigma^2 = 0$. Then the convergence holds also in probability and in $L^2$, hence the convergence holds in $L^p$ for every $p > 0$.
\end{lemma}
\begin{proof}[Proof of Theorem \ref{thm:thm6}]
	By definition, the difference $\eta(t) - \eta_n(t)$ has a Gaussian distribution with zero mean and variance
	\begin{multline*}
		\delta^2_n = 
		\int_{-\infty}^t \left[ G_\rho^\mu(-(t-u)^\rho/\lambda) - \frac1n \sum_{k=1}^n s_{\alpha_k}(t-u) \right]^2 \, {\rm d}u \\
		= \int_0^\infty \left[ \bar{\mathbb{E}}[s_\alpha(u)] -  \frac1n \sum_{k=1}^n s_{\alpha_k}(u) \right]^2 \, {\rm d}u.
	\end{multline*}
	Taking into account the results in Lemma \ref{le:Baldi}, the claim of the theorem follows by proving that
	\begin{equation}
		\label{eq:da_provare}
		\lim_{n \to \infty} \int_0^\infty \left| \frac{1}{n} \sum_{k=1}^n s_{\alpha_k}(t) - \bar{\mathbb{E}}[s_{\alpha}(t)] \right|^2 \, {\rm d}t = 0
	\end{equation}
	holds $\bar{\mathbb{P}}$-a.s.
	
	As opposed to Theorem~\ref{thm:l2-conv-Yn}, here we deal with an infinite horizon convergence problem.
	However,
	the uniform law of large numbers  \cite[Lemma 2.4]{NewMcF1994HoE} implies that for every $m \in \mathbb{N}$ there exists a negligible set $N_m \subset \bar\Omega$ such that the convergence
	\begin{equation*}
		\sup_{t \in [m,m+1]} \left| \frac{1}{n} \sum_{k=1}^n s_{\alpha_k}(t) - \bar{\mathbb{E}}[s_\alpha(t)] \right| \to 0 
	\end{equation*}
	holds outside $N_m$.
	
	By setting $N = \bigcup_m N_m$ we have $\bar{\mathbb{P}}(N) = 0$ and the pointwise convergence 
	\begin{equation*}
		\left| \frac{1}{n} \sum_{k=1}^n s_{\alpha_k}(t) - \bar{\mathbb{E}}[s_\alpha(t)] \right| \to 0 
	\end{equation*}
	for $t \ge 0$, $\bar{\mathbb{P}}$-a.s.
	
	It remains to prove a dominated convergence theorem which holds $\bar{\mathbb{P}}$-a.s.
	For this purpose, let us define
	\begin{equation*}
		f_n(t) = \frac{1}{n} \sum_{k=1}^n s_{\alpha_k}(t),
	\end{equation*}
	which satisfies, by estimate \eqref{eq:ML-bdd},
	\begin{equation*}
		|f_n(t)| \le \frac{1}{n} \sum_{k=1}^n \frac{M}{1 + \alpha_k t^\rho}.
	\end{equation*}
	For simplicity, we split the estimate in two parts as follows:
	\begin{align*}
		|f_n(t)| \le M, \qquad t \in [0,1],
	\end{align*}
	and, for some $\varepsilon > 0$ to be chosen later, an explicit computation 
	leads to
	\begin{align*}
		|f_n(t)| \le C(M, \varepsilon, \rho) \frac{1}{n}\sum_{k=1}^n \alpha_k^{-(1+\varepsilon)/2\rho} \, t^{-(1 + \varepsilon)/2}, \qquad t \ge 1
	\end{align*}
	where $C(M, \varepsilon, \rho)$ is a constant depending only on the stated quantities.
	The sequence of random variables $\{\alpha_k^{-(1+\varepsilon)/2\rho}\}$ satisfies the law of large numbers provided that each of its term is integrable, i.e.~under the condition
	\begin{equation*}
		0< \varepsilon < 2\mu\rho - 1.
	\end{equation*}
	This is a nonempty interval thanks to the inequality in \eqref{eq:condition9}.
	
	Therefore, outside an event of zero measure we have, eventually in $n$,
	\begin{align*}
		|f_n(t)| \le f(t) \coloneqq C(M, \varepsilon, \rho) \left( \bar{\mathbb{E}}[\alpha^{-(1+\varepsilon)/2\rho}] + 1 \right) \, t^{-(1 + \varepsilon)/2}, \qquad t \ge 1
	\end{align*}
	and the function $f(t)$ is square integrable on $\mathbb{R}_+$.
	
	In order to conclude the proof, it remains to apply the dominated convergence theorem to the integrals in \eqref{eq:da_provare}. 
	The sequence of functions
	\begin{align*}
		\left| \frac{1}{n} \sum_{k=1}^n s(\alpha_k, t) - \bar{\mathbb{E}}[s(\alpha, t)] \right|^2
	\end{align*}
	converges pointwise to 0 and its dominated by the integrable function
	\begin{align*}
		|f(t)|^2 + |G^\mu_\rho(-t^\rho/\lambda)|^2
	\end{align*}
	(compare with \eqref{eq:G-is-square-int}). Then Lebesgue's Dominated Convergence Theorem implies that the limit in \eqref{eq:da_provare} holds.
\end{proof}

The next, and final, step is to prove that the weak convergence of the sequence of processes $\eta_n$ to $\eta$ holds in the space $C([0,T])$ of continuous functions on $[0,T]$, for fixed $T > 0$.

\begin{theorem}
	The sequence $\eta_n$ $\bP$-weakly converges to the process $\eta$ in the space $C([0,T])$, $\bar\bP$-almost surely.
\end{theorem}

Notice that the proof is quite similar to the one of Theorem~\ref{th:weak-conv}: a consequence of the convergence of finite-dimensional distributions and tightness of the laws. In the sequel we shall only emphasise the differences in computations.

The convergence of the finite dimensional distributions follows from the fact that all the processes involved are Gaussian and thus we have established $L^2$ and pointwise convergence of $\eta_n(t)$ to $\eta(t)$ for all $t$, as in Lemma~\ref{le:finite-distrib}.

The next result is the analogue of Lemma~\ref{le:tight} and deals with the tightness estimate corresponding to \eqref{eq:tight}.
\begin{lemma}
	There exist constants $\delta, \kappa > 0$ such that for all $s < t \in [0,T]$
	\begin{equation*}
		\bE[|\eta_n(t) - \eta_n(s)|^\kappa] \le C(\kappa, \delta) |t-s|^{1+\delta},
	\end{equation*}
	where $C(\kappa, \delta)$ depends on the constants $\kappa$ and $\delta$ only.
\end{lemma}

\begin{proof}
	By an analysis of the proof of Lemma~\ref{le:tight}, we see that the key point is to estimate the quantity
	\begin{equation*}
		\begin{split}
			\bE[|\xi(t) -& \xi(s)|^2] \\&=  \bE \left| \int_{-\infty}^s [s_\alpha(t-u) - s_\alpha(s-u)] \, \mathrm{d}W(u) + \int_s^t s_\alpha(t-u) \, \mathrm{d}W(u) \right|^2
			\\
			&\le  2 \int_0^\infty [s_\alpha(t-s+u) - s_\alpha(u)]^2 \, \mathrm{d}u + 2 \int_0^{t-s} [s_\alpha(u)]^2 \, \mathrm{d}u
		\end{split}
	\end{equation*}
	and, in particular, the first term, where now the integration's interval is~$\bR_+$.
	As far as the first term is concerned, we get
	\begin{equation*}
		\int_0^\infty [s_\alpha(t-s+u) - s_\alpha(u)]^2 \, \mathrm{d}u
		= \alpha^{-1/\rho} \int_0^{\infty}  \left| \int_v^{v + \alpha^{1/\rho} (t-s)} s_1'(x) \, \mathrm{d}x \right|^2 \, \mathrm{d}v.
	\end{equation*}
	We recall the bound
	\begin{equation*}
		\left| s'_1(x) \right| \le x^{\rho - 1} \frac{M_2}{1 + x^\rho} \le \frac{M_3}{1 + x},
	\end{equation*}
	and divide the integral in two parts, for small $v$ (say, $v \in [0,1]$) and for large $v$ (say $v \ge 1$).
	We thus estimate the right hand side of the previous estimate by
	\begin{multline*}
		R = \alpha^{-1/\rho} \int_0^1 \left| M_3 \alpha^{1/\rho}(t-s) \right|^2 \, \mathrm{d}v\\
		+ M_3^2 \alpha^{-1/\rho} \int_1^\infty \left| \log(v+ \alpha^{1/\rho}(t-s)) - \log(v) \right|^2 \, \mathrm{d}v.
	\end{multline*}
	Recalling that $\log(1+x) \le x$ for all $x > 0$,
	\begin{equation*}
		R \le M_3^2 \alpha^{1/\rho} (t-s)^2 + M_3^2 \alpha^{-1/\rho} \int_1^\infty \left| \frac{\alpha^{1/\rho}(t-s)}{v} \right|^2 \, \mathrm{d}v
		\le 2M_3^2 \alpha^{1/\rho} (t-s)^2.
	\end{equation*}
	In conclusion, we get
	$	\bE[|\xi(t) - \xi(s)|^2]
	\le [2 M_3^2 \alpha^{2/\rho} T^2 + 2 M^2] (t-s)$,
	which further implies
	$	\bE[|\eta_n(t) - \eta_n(s)|^2] \le K_n (t-s)$
	and, finally, we obtain
	\begin{equation*}
		\bE[|\eta_n(t) - \eta_n(s)|^{2n}] \le C_{n} |t-s|^{n}. \qedhere
	\end{equation*}
\end{proof}

\section*{Acknowledgments}
	We would like to thank prof.~Luisa Beghin for a fruitful discussion and for pointing us towards Wright-Fox generalized functions, as well as the anonymous referees for their comments and remarks, which helped us in improving the paper.


\begin{thebibliography}{10}

\bibitem{Mathematica}
Mathematica, {{Version}} 12.2.
\newblock Wolfram Research, Inc., 2020.

\bibitem{Bal2017}
P.~Baldi.
\newblock {\em Stochastic Calculus: An Introduction through Theory and
	Exercises}.
\newblock Universitext. {Springer}, {Cham}, 2017.

\bibitem{BarBenVer2011AMMfF}
O.~E.~{Barndorff-Nielsen}, F.~E.~Benth, and A.~E.~D.~Veraart.
\newblock Ambit {{Processes}} and {{Stochastic Partial Differential
		Equations}}.
\newblock In G.~Di~Nunno and B.~{\O}ksendal, editors, {\em Advanced
	{{Mathematical Methods}} for {{Finance}}}, pages 35--74. {Springer}, {Berlin,
	Heidelberg}, 2011.

\bibitem{BarSch2007SAA}
O.~E.~{Barndorff-Nielsen} and J.~Schmiegel.
\newblock Ambit {{Processes}}; with {{Applications}} to {{Turbulence}} and
{{Tumour Growth}}.
\newblock In F.~E.~Benth, G.~Di~Nunno, T.~Lindstr{\o}m, B.~{\O}ksendal, and
T.~Zhang, editors, {\em Stochastic {{Analysis}} and {{Applications}}}, Abel
{{Symposia}}, pages 93--124, {Berlin, Heidelberg}, 2007. {Springer}.

\bibitem{BarSch2009AFM}
O.~E.~{Barndorff-Nielsen} and J.~Schmiegel.
\newblock Brownian semistationary processes and volatility/intermittency.
\newblock In {\em Advanced {{Financial Modelling}}}, pages 1--26. {De Gruyter},
Dec. 2009.

\bibitem{Bil1968}
P.~Billingsley.
\newblock {\em Convergence of Probability Measures}.
\newblock {John Wiley \& Sons, Inc., New York-London-Sydney}, first edition,
1968.

\bibitem{BonTub2003SAaA}
S.~Bonaccorsi and L.~Tubaro.
\newblock Mittag-{{Leffler}}'s {{Function}} and {{Stochastic Linear Volterra
		Equations}} of {{Convolution Type}}.
\newblock {\em Stochastic Analysis and Applications}, 21(1):61--78, Jan. 2003.

\bibitem{CleDa1996ADANDLCSFMENRLMEA}
P.~Cl{\'e}ment and G.~Da~Prato.
\newblock Some results on stochastic convolutions arising in {{Volterra}}
equations perturbed by noise.
\newblock {\em Atti della Accademia Nazionale dei Lincei. Classe di Scienze
	Fisiche, Matematiche e Naturali. Rendiconti Lincei. Matematica e
	Applicazioni}, 7(3):147--153, 1996.

\bibitem{DouEs-Tud2020PMD}
S.~Douissi, K.~{Es-Sebaiy}, and C.~A.~Tudor.
\newblock Hermite {{Ornstein--Uhlenbeck}} processes mixed with a {{Gamma}}
distribution.
\newblock {\em Publicationes Mathematicae Debrecen}, 96(1-2):23--44, Jan. 2020.

\bibitem{ErdMagObeTri1955}
A.~Erd{\'e}lyi, W.~Magnus, F.~Oberhettinger, and F.~Tricomi.
\newblock {\em Higher Trascendental Functions}, volume III.
\newblock {McGraw-Hill}, {New York}, 1955.

\bibitem{Es-FarHil2020PMS}
K.~{Es-Sebaiy}, F.-E.~Farah, and A.~Hilbert.
\newblock Weyl multifractional {{Ornstein}}\textendash{{Uhlenbeck}} processes
mixed with a {{Gamma}} distribution.
\newblock {\em Probability and Mathematical Statistics}, 40(2):269--295, 2020.

\bibitem{Es-Tud2015F}
K.~{Es-Sebaiy} and C.~A.~Tudor.
\newblock Fractional {{Ornstein-Uhlenbeck Processes Mixed}} with a {{Gamma
		Distribution}}.
\newblock {\em Fractals}, 23(03):1550032, Sept. 2015.

\bibitem{GorKilMaiRog2020}
R.~Gorenflo, A.~A.~Kilbas, F.~Mainardi, and S.~Rogosin.
\newblock {\em Mittag-{{Leffler Functions}}, {{Related Topics}} and
	{{Applications}}}.
\newblock Springer {{Monographs}} in {{Mathematics}}. {Springer}, {Berlin,
	Heidelberg}, 2020.

\bibitem{IglTer1999EJP}
E.~Igloi and G.~Terdik.
\newblock Long-range {{Dependence}} trough {{Gamma-mixed Ornstein-Uhlenbeck
		Process}}.
\newblock {\em Electronic Journal of Probability}, 4, 1999.

\bibitem{Luc2019BT}
Y.~Luchko.
\newblock The {{Wright}} function and its applications.
\newblock In A.~Kochubei and Y.~Luchko, editors, {\em Basic {{Theory}}}, pages
241--268. {De Gruyter}, Feb. 2019.

\bibitem{Mit1905AM}
G.~{Mittag-Leffler}.
\newblock Sur la repr\'esentation analytique d'une branche uniforme d'une
fonction monog\`ene: Cinqui\`eme note.
\newblock {\em Acta Mathematica}, 29:101--181, Jan. 1905.

\bibitem{NewMcF1994HoE}
W.~K.~Newey and D.~McFadden.
\newblock Large {{Sample Estimation}} and {{Hypothesis Testing}}.
\newblock In R.~F.~Engle and D.~McFadden, editors, {\em Handbook of
	{{Econometics}}}, volume~IV, pages 2111--2245. {Elsevier science B.V.}, 1994.

\bibitem{PopSed2013JMS}
A.~Y.~Popov and A.~M.~Sedletskii.
\newblock Distribution of roots of {{Mittag-Leffler}} functions.
\newblock {\em Journal of Mathematical Sciences}, 190(2):209--409, Apr. 2013.

\bibitem{Rie1949AM}
M.~Riesz.
\newblock {L'Int\'egrale de Riemann-Liouville et le probl\'eme de Cauchy}.
\newblock {\em Acta Mathematica}, 81:222, 1949.

\bibitem{TriErd1951PJM}
F.~Tricomi and A.~Erd{\'e}lyi.
\newblock The asymptotic expansion of a ratio of gamma functions.
\newblock {\em Pacific Journal of Mathematics}, 1(1):133--142, 1951.

\end{thebibliography}

\end{document}